\documentclass[12pt]{article}

\usepackage{amsmath,amsthm,amssymb}
\allowdisplaybreaks
\usepackage{mathrsfs}
\usepackage{cite}
\usepackage{hyperref}

\newtheorem{theorem}{Theorem}[section]

\newtheorem{lemma}[theorem]{Lemma}

\theoremstyle{remark}

\theoremstyle{remark}

\theoremstyle{remark}
\newtheorem{remark}[theorem]{Remark}

\makeatletter\@addtoreset{equation}{section}\makeatother

\newcommand{\Sym}{\operatorname{Sym}}
\newcommand{\R}{\mathbb R}
\begin{document}

\vspace{-20mm}
\begin{center}{\Large \bf
A chaotic decomposition for generalized stochastic processes with independent values}\end{center}

{\large Suman Das }\\ Department of Mathematics,
Swansea University, Singleton Park, Swansea SA2 8PP, U.K.;
e-mail: \texttt{380535@swansea.ac.uk}\vspace{2mm}

{\large Eugene Lytvynov}\\ Department of Mathematics,
Swansea University, Singleton Park, Swansea SA2 8PP, U.K.;
e-mail: \texttt{e.lytvynov@swansea.ac.uk}\vspace{2mm}

\begin{center}
{\bf Abstract}
\end{center}
{\small \noindent 
We extend the result of Nualart and Schoutens on chaotic decomposition of the $L^2$-space of a L\'evy process to the case of a generalized stochastic processes with independent values.}

\section{Introduction} Among all stochastic processes with independent increments, essentially only Brownian motion and Poisson process have a chaotic representation  property. The latter property means that, by using multiple stochastic integrals with respect to the centered stochastic process, one can construct a  unitary isomorphism between the $L^2$-space of the process and a  symmetric Fock space.
In the case of a L\'evy process, several approaches have been proposed in order to construct a Fock space-type realization of the corresponding $L^2$-space. In this paper, we will be concerned with the approach of Nualart and Schoutens  \cite{NS}, who constructed a representation of every square integrable functional of a L\'evy process
in terms of orthogonalized Teugels martingales. 
Recall that, for a given L\'evy process $(X_t)_{t\ge0}$, its $k$-th order Teugels martingale is defined by centering the power jump process
$$X_t^{(k)}:=\sum_{0<s\le t}(\Delta X_s)^k, \quad k\in\mathbb N.$$
For numerous applications of this result, see e.g.\ \cite{Oksendal,Schoutens}. We also refer to  \cite{Lin} for an extension of this result to the case of a L\'evy process taking values in $\R^d$, and to \cite{A,BL} for a Nualart--Schotens-type decomposition for noncommutative (in particular, free) L\'evy processes.

The aim of this note is to extend the Nualart--Schoutens decomposition to  the case of a generalized stochastic process with independent values.
Consider a standard triple
$\mathcal D\subset L^2(\mathbb R^d,dx)\subset \mathcal D'$, where $\mathcal D=C_0^\infty(\mathbb R^d)$ is the nuclear space of all smooth, compactly supported functions on $\R^d$, and $\mathcal D'$ is the dual space of $\mathcal D$ with respect to the center space $L^2(\R^d,dx)$, see e.g.\ \cite{BK} for detail. For $\omega\in\mathcal D'$ and $\varphi\in\mathcal D$, we denote by $\langle\omega,\varphi\rangle$ the dual pairing of $\omega$ and $\varphi$. 
Denote by $\mathcal C(\mathcal D')$ the cylinder $\sigma$-algebra on $\mathcal D'$. A generalized stochastic process is a probability measure $\mu$ on $(\mathcal D',\mathcal C(\mathcal D'))$. Thus, a generalized stochastic process is a random generalized function $\omega\in\mathcal D'$. One says that a generalized stochastic process has independent values, if for any $\varphi_1,\dots,\varphi_n\in\mathcal D$ which have mutually disjoint support, the random variables $\langle\omega,\varphi_1\rangle,\dots,\langle\omega,\varphi_n\rangle$ are independent. So, heuristically, we have that, for any $x_1,\dots,x_n\in\R^d$, the random variables $\omega(x_1),\dots,\omega(x_n)$ are independent. In the case where $d=1$, one can (at least heuristically) interpret $\omega(t)$ as the time $t$ derivative of a classical stochastic process $X=(X(t))_{t\in\R}$ with independent increments, so that, for $t\ge0$,  $X(t)=\int_0^t\omega(s)\,ds$. 

If a generalized stochastic process with independent values, $\mu$, has the property 
that the measure $\mu$ remains invariant under each transformation $x\mapsto x+a$ ($a\in\R^d$) of the underlying space, then one calls $\mu$ a L\'evy process (which is, for
 $d=1$, the time derivative of a classical L\'evy process.)
So, below, for a certain class of generalized stochastic processes with independent values, we will construct an orthogonal decomposition of the space $L^2(\mathcal D',\mu)$, which, in the case of a classical L\'evy process, will be exactly the Nualart--Schotens decomposition from \cite{NS}. This paper will also extend the results of \cite{L} for generalized stochastic processes being L\'evy processes. 

\section{Preliminaries}
We start by briefly recalling some results from \cite{D}.
Assume that for each $x \in \mathbb {R}^d$, $\sigma(x,ds)$ is a probability measure on
$(\mathbb{R},\mathcal{B}(\mathbb{R}))$. We also assume that for each $\Delta \in \mathcal{B}(\mathbb{R})$,
$
\mathbb{R}^d \ni x\mapsto \sigma (x, \Delta)
$
is a measurable mapping.
Hence,  we can define a $\sigma$-finite measure
$dx\, \sigma(x,ds)$ on $(\mathbb{R}^d \times \mathbb{R},\mathcal{B}(\mathbb{R}^d \times \mathbb{R}))$.
Let $\mathcal{B}_0 (\mathbb{R}^d)$ denote the collection of all sets
$\Lambda \in \mathcal{B}(\mathbb{R}^d)$ which are bounded. We will
additionally assume that, for each $\Lambda \in
\mathcal{B}_0(\mathbb{R}^d)$, there exists  $C_\Lambda >0$ such that
\begin{align} \label{equ 3.2}
\int_\mathbb{R} |s|^n \sigma(x, ds) \leq C_\Lambda^n n! \quad  n \in
\mathbb{N},
\end{align}
for all $x \in \Lambda$. We fix the Hilbert space
$
H = L^2(\mathbb{R}^d \times\mathbb{R},dx\,\sigma(x,ds))$.
We denote by
$\mathcal F(H)=\bigoplus_{n=0}^\infty H^{\odot n}n!$
the symmetric Fock space over $H$. Here $\odot$ denotes symmetric tensor product. We denote by $\mathfrak D$ the subset of $\mathcal F(H)$ which consists of all finite vectors $f=(f^{(0)},f^{(1)},\dots,f^{(n)},0,0,\dots)$ where each $f^{(k)}$ is a symmetric function on $(\R^d\times\R)^k$ which is obtained as the symmetrization of a finite sum of functions of the form
$$g^{(k)}(x_1,s_1,\dots,x_k,s_k)=\phi(x_1,\dots,x_k)s_1^{i_1}\dotsm s_k^{i_k},$$
where $\phi\in \mathcal D^{\otimes k}=C_0^\infty((\R^d)^k)$ and $i_1,\dots,i_k\in\mathbb Z_+=\{0,1,2,\dots\}$. For each $\varphi\in\mathcal D$, we define an operator $A(\varphi)$ in $\mathcal F(H)$ with domain $\mathfrak D$ by
\begin{equation}\label{rdtsw6u4ebg}
A(\varphi ):= a^+(\varphi \otimes m_0) + a^-(\varphi \otimes m_0) + a^0(\varphi\otimes m_1).\end{equation}
Here and below, for $i\in\mathbb Z_+:=\{0,1,2,\dots\}$,
$$(\varphi\otimes m_i)(x,s):=\varphi(x)s^i,$$
  $a^+(\varphi \otimes m_i)$ is the creation operator corresponding to $\varphi\otimes m_i$:
$$  a^+(\varphi \otimes m_i)f^{(k)}=f^{(k)}\odot (\varphi\otimes m_i),\quad f^{(k)}\in H^{\odot k},$$
$a^-(\varphi \otimes m_i)$ is the corresponding annihilation operator:
$$  a^-(\varphi \otimes m_i)f^{(k)}=k\int_{\R^d\times\R}dy\, \sigma(y,du)\varphi(y)u^if^{(k)}(y,u,\cdot),$$
and $a^0(\varphi\otimes m_i)$ is the neutral operator corresponding to $\varphi\otimes m_i$:
\begin{multline*}\big(a^0(\varphi \otimes m_i)f^{(k)}\big)(x_1,s_1,\dots,x_k,s_k)\\=\big(\varphi(x_1)s_1^i+\dots+\varphi(x_k)s_k^i\big)f^{(k)}(x_1,s_1,\dots,x_k,s_k). \end{multline*}
Note that $A(\varphi)$ maps $\mathfrak D$ into itself, and it is a symmetric operator in $\mathcal F(H)$.

\begin{theorem}\label{ftydyr7rdaufc}
For each $\varphi\in\mathcal D$, the operator $A(\varphi)$ is essentially self-adjoint on $\mathfrak D$. Furthermore, there exists a unique probability measure $\mu$ on $\mathcal D'$ such that the linear operator $I:\mathcal F(H)\to L^2(\mathcal D',\mu)$ given through $I\Omega=1$, $\Omega$ being the vacuum vector $(1,0,0,\dots)$, and
$$ I(A(\varphi_1)\dotsm A(\varphi_n)\Omega)=\langle\omega,\varphi_1\rangle\dotsm \langle\omega,\varphi_n\rangle,$$
is a unitary operator. The Fourier transform of the measure $\mu$ is given by
\begin{equation}\label{the3.22}
\begin{split}
\int_{\mathcal{D}'} e^{i\langle \varphi,\omega \rangle}
\mu(d\omega) &= \exp\bigg[-\frac{1}{2} \int_{\mathbb{R}^d}dx\,
\sigma(x,\{0\})\varphi(x)^2  \\
&\quad+ \int_{\mathbb{R}^d} dx \int_{\mathbb{R}^*} \sigma(x, ds)
\frac{1}{s^2}(e^{i\varphi(x)s} - i\varphi(x)s - 1 )\bigg],
\end{split}
\end{equation}
where $\mathbb R^*:=\mathbb R\setminus\{0\}$. In particular, $\mu$ is a generalized stochastic process with independent values.
\end{theorem}

Note that, if the measure $\sigma(ds)=\sigma(x,ds)$ is the same for all $x\in\R^d$, then $\mu$ is a L\'evy process. 

\section{An orthogonal decomposition of a Fock space}

We will now discuss  an orthogonal decomposition of a general symmetric Fock space. This decomposition generalizes the well-known basis of occupation numbers in the Fock space, see e.g.\ \cite{BK}.

In this section, we will denote by $H$ any real separable Hilbert space. Let $(H_k)_{k=0}^\infty$ be a sequence of closed subspaces of $H$ such that
$H= \bigoplus_{k=0}^\infty H_k$. Let $n \geq 2$. Then clearly
\begin{align}
H^{\otimes n} &= \bigg(\bigoplus_{k_1=0}^\infty H_{k_1}\bigg) \otimes \bigg(\bigoplus_{k_2=0}^\infty H_{k_2}\big) \otimes \cdots
\otimes \bigg(\bigoplus_{k_n=0}^\infty H_{k_n}\bigg) \notag\\
&=\bigoplus_{(k_1,k_2, \dots ,k_n) \in \mathbb{Z}_+^n} H_{k_1} \otimes H_{k_2} \otimes \cdots \otimes H_{k_n}.\label{ns1}
\end{align}
Denote by $\operatorname{Sym}_n$ the orthogonal projection of $H^{\otimes n}$ onto $H^{\odot n}$. 
Recall that, for any
$f_1,f_2,\ldots,f_n \in H$
\begin{equation}\label{ns2}
f_1\odot \dotsm \odot f_n=\Sym_n\,f_1 \otimes \cdots \otimes f_n =\frac1{n!} \sum_{\sigma \in S_n} f_{\sigma(1)} \otimes \cdots \otimes f_{\sigma(n)}.
\end{equation} (Here, $S_n$ denotes the symmetric group of order $n$.)
For each $(k_1,k_2, \dots ,k_n) \in \mathbb{Z}_+^n$, let $H_{k_1}
\odot H_{k_2} \odot \cdots \odot H_{k_n}$ denote the Hilbert space $\Sym_n (H_{k_1} \otimes H_{k_2} \otimes \cdots \otimes H_{k_n})$, i.e.,
the space of all $\Sym_n$-projections of elements of $H_{k_1} \otimes H_{k_2} \otimes \cdots \otimes H_{k_n}$.

Assume that $(k_1,k_2, \dots ,k_n) \in \mathbb{Z}_+^n$ , $(l_1,l_2, \dots ,l_n) \in \mathbb{Z}_+^n$ are such that there exists a permutation
$\sigma \in S_n$ such that
\begin{equation}\label{ns3}
(k_1,k_2, \dots ,k_n) = (l_{\sigma (1)},l_{\sigma(2)}, \dots ,l_{\sigma(n)}).
\end{equation}
Then
\begin{align}\label{ns4}
H_{k_1} \odot H_{k_2} \odot \cdots \odot H_{k_n} = H_{l_1} \odot H_{l_2} \odot \cdots \odot H_{l_n}.
\end{align}
Indeed, take any $f_1 \in H_{l_1}, f_2 \in H_{l_2}, \ldots, f_n \in H_{l_n}$. Then
\begin{align}\label{ns5}
f_1 \odot f_2 \odot \cdots \odot f_n= f_{\sigma(1)}\odot f_{\sigma(2)}\odot  \cdots \odot  f_{\sigma(n)}.
\end{align}
We have $f_{\sigma(i)} \in H_{l_{\sigma(i)}} =H_{k_i}$. Therefore, the vector in \eqref{ns5} belongs to $H_{k_1}\odot H_{k_2}\odot  \cdots \odot  H_{k_n}$. Since the set of all vectors of the form $f_1 \odot f_2 \odot \cdots \odot f_n$ with
$f_i \in H_{l_i}$ is total in $H_{l_1} \odot  H_{l_2}\odot  \cdots \odot H_{l_n}$, we therefore conclude that
$$H_{l_1} \odot  H_{l_2}\odot  \cdots \odot H_{l_n} \subset H_{k_1} \odot  H_{k_2}\odot  \cdots \odot H_{k_n}$$
By inverting the argument, we obtain the inverse conclusion, and so formula \eqref{ns4} holds.

If no permutation $\sigma \in S_n$ exists which satisfies \eqref{ns3}, then
\begin{align}\label{ns6}
H_{k_1} \odot H_{k_2} \odot \cdots \odot H_{k_n} \bot \; H_{l_1} \odot H_{l_2} \odot \cdots \odot H_{l_n}.
\end{align}
Indeed, take any $f_i \in H_{k_i}$, $g_i \in H_{l_i}$, $i = 1, 2, \ldots, n$. Then, since $\Sym_n$ is an orthogonal projection,
\begin{align*}
&\big(f_1\odot f_2 \odot  \cdots \odot  f_n, g_1 \odot  g_2 \odot  \cdots \odot  g_n\big)_{H^{\odot n}} \\
& \quad = \bigg( \Sym_n \big(f_1\otimes f_2 \otimes  \cdots \otimes  f_n\big), g_1 \otimes  g_2 \otimes  \cdots \otimes  g_n\bigg)_{H^{\otimes n}}\\
& \quad = \frac{1}{n!}\sum_{\sigma \in S_n} \prod_{i=1}^n (f_{\sigma(i)}, g_i)_H  = \frac{1}{n!}\sum_{\sigma \in S_n} \prod_{i=1}^n (f_i, g_{\sigma(i)})_H = 0.
\end{align*}
Since the vectors of the form $f_1\odot f_2 \odot  \cdots \odot  f_n$ with $f_i \in H_{k_i}$ and
$g_1 \odot  g_2 \odot  \cdots \odot  g_n$ with $g_i \in H_{l_i}$ form a total set in
$H_{k_1} \odot  H_{k_2} \odot  \cdots \odot  H_{k_n}$ and $H_{l_1} \odot  H_{l_2} \odot  \cdots
\odot  H_{l_n}$, respectively, we get \eqref{ns6}.

By \eqref{ns1}, the closed linear span of the spaces $H_{k_1} \odot  H_{k_2} \odot  \cdots \odot  H_{k_n}$ with $(k_1, k_2, \ldots , k_n) \in \mathbb Z_+^n$ coincides with $H^{\odot n}$. Hence, by \eqref{ns4}
and \eqref{ns6}, we get the orthogonal decomposition
\begin{align}\label{nu1}
H^{\odot n}=\bigoplus_{\alpha \in \,\mathbb{Z}_{+,0}^\infty,\, |\alpha|=n} H_0^{\odot \alpha_0} \odot H_1^{\odot \alpha_1} \odot H_2^{\odot \alpha_2}\odot \cdots .
\end{align}
Here $\mathbb{Z}_{+,0}^\infty$
denotes the set of indices $\alpha=(\alpha_0,\alpha_1,\alpha_2,\dots)$ such that all $\alpha_i\in\mathbb Z_+$ and 
$|\alpha|:=\alpha_0+\alpha_1+\alpha_2+\dotsm<\infty$.
Hence, by (\ref{nu1}), we get the following

\begin{lemma} \label{lem1.11}
We have the orthogonal decomposition of the symmetric Fock space $\mathcal F (H)=\bigoplus_{n=0}^\infty H^{\odot n}n!$\,:
\begin{equation}\label{ur6u7}
\mathcal{F}(H) =\bigoplus_{\alpha \in \mathbb{Z}_{+,0}^\infty}\big(H_0^{\odot \alpha_0} \odot H_1^{\odot \alpha_1} \odot H_2^{\odot \alpha_2} \cdots \big)|\alpha|!\,.
\end{equation} 
\end{lemma}

Next, we have:

\begin{lemma}\label{lem2}
Let  $\alpha \in \mathbb{Z}_{+,0}^\infty$\,. Then 
\begin{align}\label{ns7}
&\Sym_{|\alpha|}: \big(H_0^{ \odot \alpha_0} \otimes H_1^{\odot  \alpha_1} \otimes H_2^{\odot  \alpha_2} \otimes \cdots \big){\alpha_0! \alpha_1! \alpha_2! \cdots} \nonumber\\
& \quad \rightarrow \big(H_0^{ \odot \alpha_0} \odot H_1^{\odot  \alpha_1} \odot H_2^{\odot  \alpha_2} \odot \cdots \big)\, |\alpha|!
\end{align}
is a unitary operator.
\end{lemma}

\begin{proof} We start the proof with the following well-known observation.
Let $k, l \geq 1$, $n: = k + l$. Then
$\Sym_n = \Sym_n (\Sym_k \otimes \Sym_l)$. Hence, for any $\alpha \in \mathbb{Z}_{+,0}
^\infty$, $|\alpha| = n$, we get $\Sym_n = \Sym_n (\Sym_{\alpha_0} \otimes \Sym_{\alpha_1} \otimes \Sym_{\alpha_2} \otimes \cdots)$.
Therefore, we have the following equality of subspaces of $H^{\otimes n}:$
\begin{align*}
&H^{\odot \alpha_0}_0 \odot H_1^{\odot \alpha_1} \odot H_2^{\odot \alpha_2} \odot \cdots\\
&\quad =
\Sym_n \big(H^{\otimes \alpha_0}_0 \otimes H_1^{\otimes \alpha_1} \otimes H_2^{\otimes \alpha_2} \otimes \cdots\big)\\
& \quad = \Sym_n \big(\Sym_{\alpha_0} \otimes \Sym_{\alpha_1} \otimes \Sym_{\alpha_2} \otimes \cdots
\big)\big(H^{\otimes \alpha_0}_0 \otimes H_1^{\otimes \alpha_1} \otimes H_2^{\otimes \alpha_2} \otimes \cdots\big)\\
&\quad = \Sym_n \big(H^{\odot \alpha_0}_0 \otimes H_1^{\odot \alpha_1} \otimes H_2^{\odot \alpha_2} \otimes \cdots\big).
\end{align*}
This shows that the image of the operator $\Sym_n$ in \eqref{ns7} is the whole space $H_0^{ \odot \alpha_0} \odot H_1^{\odot  \alpha_1} \odot H_2^{\odot  \alpha_2} \odot \cdots \, n!$\,. Hence, we only need to prove that this operator is an isometry.

 Fix any $f_i, g_i \in H_i$ with $i\in\mathbb Z_+$ and any $\alpha \in \mathbb{Z}_{+,0}^\infty$. Then, by \eqref{ns2}
 \begin{align*}\label{ns8}
 & \big(\Sym_n \big(f_0^{\otimes \alpha_0} \otimes f_1^{\otimes \alpha_1} \otimes f_2^{\otimes
\alpha_2} \otimes \cdots\big), \Sym_n \big(g_0^{\otimes \alpha_0} \otimes g_1^{\otimes \alpha_1} \otimes g_2^{\otimes
\alpha_2} \otimes \cdots\big)\big)_{H^{\odot n}} 
\nonumber\\
&\quad =\big(\Sym_n \big(f_0^{\otimes \alpha_0} \otimes f_1^{\otimes \alpha_1} \otimes f_2^{\otimes
\alpha_2} \otimes \cdots\big), g_0^{\otimes \alpha_0} \otimes g_1^{\otimes \alpha_1} \otimes g_2^{\otimes
\alpha_2} \otimes \cdots\big)_{H^{\otimes n }} 
\nonumber\\
& \quad =\frac1{n!}\sum_{\sigma_0 \in S_{\alpha_0}} \big( f_0, g_0\big)_{H_0}^{\alpha_0}.\sum_{\sigma_1 \in S_{\alpha_1}} \big( f_1, g_1\big)_{H_1}^{\alpha_1} \cdots \nonumber\\
&\quad = \frac1{n!}\big(f_0^{\otimes \alpha_0}, g_0^{\otimes \alpha_0}\big)_{H_0^{\odot \alpha_0}} \alpha_0!
\big(f_1^{\otimes \alpha_1}, g_1^{\otimes \alpha_1}\big)_{H_1^{\odot \alpha_1}} \alpha_1! \cdots\nonumber \\
&\quad =\frac1{n!} \big(f_0^{\otimes \alpha_0} \otimes f_1^{\otimes \alpha_1} \otimes \cdots ,\, g_0^{\otimes \alpha_0} \otimes g_1^{\otimes \alpha_1}
\otimes \cdots\big)_{H_0^{\odot  \alpha_0}\otimes H_1^{\odot  \alpha_1} \otimes \cdots}\alpha_0!\alpha_1!\cdots\,.
\end{align*}
Since the set of all vectors of the form $f_i^{\otimes \alpha_i}$ with $f_i \in H_i$ is a total subset of $H_i^{\odot \alpha_i}$,
we conclude that the operator in \eqref{ns7} is indeed an isometry.\end{proof}

We define the symmetrization operator
\begin{equation}\label{yutr67r76r}\Sym: \bigoplus_{\alpha\in \mathbb Z_{+,0}^\infty} \big(H_0^{ \odot \alpha_0} \otimes H_1^{\odot  \alpha_1} \otimes H_2^{\odot  \alpha_2} \otimes \cdots \big){\alpha_0! \alpha_1! \alpha_2! \cdots}\to\mathcal F(H)\end{equation}
so that the restriction of $\Sym$ to each space 
$$\big(H_0^{ \odot \alpha_0} \otimes H_1^{\odot  \alpha_1} \otimes H_2^{\odot  \alpha_2} \otimes \cdots \big){\alpha_0! \alpha_1! \alpha_2! \cdots}$$ is equal to $\Sym_{|\alpha|}$.
By Lemmas \ref{lem1.11} and  \ref{lem2}, we get

\begin{lemma}\label{gfut7r7}
The symmetrization operator $\Sym$ is a unitary operator.

\end{lemma}

\begin{remark}
Let us assume that each Hilbert space $H_k$ is
one-dimensional and in each $H_k$ we fix a vector $e_k \in H_k$ such that $||e_k||= 1$.
Thus, $(e_k)_{k=0}^\infty$ is an orthonormal basis of $H$. By Lemma \ref{gfut7r7},  the set of the vectors $$
\left(\big(\alpha_0!\alpha_1!\alpha_2! \cdots\big)^{-\frac{1}{2}}e_0^{\otimes \alpha_0} \odot e_1^{\otimes \alpha_1} \odot e_2^{\otimes \alpha_2} \odot \cdots\right)_{\alpha \in \mathbb{Z}_{+,0}^\infty}$$
is an orthonormal basis of $\mathcal F(H)$. This basis  is called a basis of occupation numbers.

\end{remark}

\section{An orthogonal decomposition of $L^2(\mathcal D',\mu)$}

We want to apply the general result about the orthogonal decomposition of the Fock space to the case of $\mathcal{F}(H)$,
where $H = L^2 (\mathbb{R}^d \times \mathbb{R}, dx \,\sigma(x,ds))$.
We note that, by \eqref{equ 3.2}, for each $x\in\R^d$, the set of polynomials is dense in $L^2(\R,\sigma(x,ds))$.
We denote by $(q^{(n)}(x,s))_{n\ge0}$
the sequence of monic polynomials which are orthogonal
 with respect to the
measure $\sigma(x,ds)$.
These polynomials satisfy the following recursive formula:
\begin{equation}
\begin{split}
&sq^{(n)}(x,s) = q^{(n+1)}(x,s) + b_n(x)q^{(n)}(x,s) + a_n(x)q^{(n-1)}(x,s),\quad n \geq 1,\\
&sq^{(0)}(x,s) = q^{(1)}(x,s) + b_0(x)
\end{split}\label{futyr457iio}
\end{equation}
with some $b_n(x)\in\R$ and $a_n(x)>0$. 
[Note that if the support of $\sigma(x,ds)$ consists of $k<\infty$ points, then, for $n\ge k$, we set $q^{(n)}(x,s)=0$, $a_n(x)=0$ with $b_n(x)\in\R$ being arbitrary.]

From now on, we will assume that the following condition is satisfied:

\begin{itemize}
\item[(A)] For each $n\in\mathbb N$, the function $a_n(x)$ from \eqref{futyr457iio} is locally bounded on $\mathbb R^d$, i.e., for each $\Lambda\in\mathcal B_0(\mathbb R^d)$,
$\sup_{x\in\Lambda}a_n(x)<\infty$.
\end{itemize}

Denote by $\mathfrak L$ the linear space of all functions on $\mathbb R^d \times \mathbb R$ which have the form
\begin{equation}\label{fty6r6}
f(x,s) = \sum_{k=0}^n a_k(x) q^{(k)}(x,s),\end{equation} where $n \in \mathbb N$, $a_k \in \mathcal D$, $k=0,1,\ldots,n$. 

\begin{lemma}\label{giyr8i8}
The space $\mathfrak L$ is densely embedded into $H$.
\end{lemma}

\begin{proof} Let $f(x,s)=a(x)q^{(k)}(x,s)$, where $a\in\mathcal D$.   Let us show that $f\in H$.
Denote $\Lambda:=\operatorname{supp}(a)$.
We have, for some $C>0$,
\begin{equation}\label{76r8}\int_{\mathbb R^d}\int_{\mathbb R}dx\,\sigma(x,ds)f(x,s)^2\le C\int_\Lambda dx\int_{\mathbb R}\sigma(x,ds)\,q^{(k)}(x,s)^2. \end{equation}
If $k = 0$, then $q^{(0)} (x,s) =1$, and the right hand side of \eqref{76r8} is evidently finite.
By the theory of orthogonal polynomials (see e.g. \cite{Chihara})
\begin{equation}\label{fr6de6}\int_{\mathbb{R}} \sigma(x,ds)\,q^{(k)}(x,s)^2 = a_1(x)a_2(x) \cdots a_k(x), \quad k \geq1.\end{equation}
Hence we continue \eqref{76r8}
$$\le  C\int_\Lambda dx\,a_1(x)a_2(x) \cdots a_k(x)<\infty$$
by (A). Thus, $\mathfrak L\subset H$.

We now have to show that $\mathfrak L$ is a dense subset of $H$. Let $g\in H$
be such that $(g,f )_{H} = 0$ for all $f \in \mathfrak{L}$.
Hence for any $a \in \mathcal{D}$ and $k \geq 0$
$$\int_{\mathbb{R}^d}dx \int_\mathbb{R} \sigma(x,ds) \, g(x,s)\, a(x)\, q^{(k)}(x,s) =0.$$
Fix any compact set $\Lambda$ in $\mathbb{R}^d$ and let $a \in \mathcal{D}$ be such that the support of $a$ is a subset of $\Lambda$. Then,
$$\int_{\mathbb{R}^d}dx\, a(x)\left( \int_\mathbb{R} \sigma(x,ds) \, g(x,s)\,\, q^{(k)}(x,s)\right) =0.$$
Hence
\begin{equation}\label{vfufs}
\int_\Lambda \,dx\, a(x)\left(\int_\mathbb{R} \sigma(x,ds) \, g(x,s) \, q^{(k)}(x,s)\right) =0.
\end{equation}
We state that the function
$$\Lambda \ni x \mapsto \int_\mathbb{R} \sigma(x,ds) \, g(x,s)\,q^{(k)}(x,s)$$
belongs to $L^2 (\Lambda, \, dx)$. Indeed, if $k = 0$, then $q^{(0)} (x,s) =1$ and this statement evidently follows from Cauchy's
inequality. Assume that $k \geq 1$.
Then by Cauchy's inequality, \eqref{76r8}, and condition (A),
\begin{align*}
 &\int_\Lambda \,dx \bigg(\int_\mathbb{R} \sigma(x,ds) \, g(x,s) \,
 q^{(k)}(x,s)\bigg)^2\\
& \quad \leq\int_\Lambda \,dx \int_\mathbb{R} \sigma(x,ds_1) \, g(x,s_1)^2 \, \int_{\mathbb{R}} \sigma(x,ds_2)\,q^{(k)}(x,s_2)^2\\
 &=\int_\Lambda dx \int_\mathbb R \sigma(x, ds)\, g(x,s)^2 \, a_1(x)a_2(x) \cdots a_k(x)\\
&\leq \left( \prod_{i=1}^k \sup_{x \in \Lambda} a_i(x)\right)
\int_\Lambda dx \int_\mathbb R \sigma(x,ds)\, g(x,s)^2 < \infty.
\end{align*}

Since the set of all functions $a \in \mathcal{D}$ with support in $\Lambda$ is dense in $L^2(\Lambda, dx)$, we
therefore conclude from \eqref{vfufs} that, for $dx$-a.a $x \in \Lambda$,
\begin{equation}\label{eq111.1}
\int_\mathbb{R} \sigma(x,ds) \, g(x,s)\, q^{(k)}(x,s) =0, \quad \forall  k \geq 0.
\end{equation}
Since $g \in H$, we get that, for $dx$-a.a. $x \in \mathbb{R}^d$, $g(x,\cdot) \in L^2(\mathbb{R}, \sigma(x,ds))$.
Since $\{q^{(k)}(x,\cdot)\}_{k=0}^\infty$ form an orthogonal basis in $L^2(\mathbb R, \sigma(x,ds))$, we conclude from \eqref{eq111.1}
that for $dx$-a.a. $x \in \mathbb R^d$ $g(x,s) =0$ for $ \sigma(x,ds)$-a.a. $s \in  \mathbb{R}$.
From here, we easily conclude  that $g=0$ as an element of $H$. Hence $\mathfrak L$ is indeed dense in $H$.
 \end{proof}

For each $n \in \mathbb Z_+$, we define $$\mathfrak L_n: =\big\{g_n(x,s) = f(x) \, q^{(n)}(x,s)  \mid f \in \mathcal D \big\}.$$
We have $\mathfrak L_n \subset \mathfrak L$, and the linear span of the $\mathfrak L_n$ spaces coincides with $\mathfrak L$.
 For any $g_n(x,s) = f_n(x) \, q^{(n)}(x,s) \in \mathfrak L_n$ and $g_m(x,s) = f_m(x) \, q^{(m)}(x,s)\in \mathfrak L_m$, $n, m \in \mathbb Z_+$, we have
\begin{equation}\label{ns10}
\begin{split}
(g_n, g_m)_H & = \int_{\mathbb {R}^d \times \mathbb {R}} g_n(x,s)\, g_m(x,s) dx \,\sigma (x,ds) \\
& = \int_{\mathbb {R}^d}f_n(x)\, f_m(x) \bigg(\int_\mathbb {R} q^{(n)} (x,s)\, q^{(m)}(x,s)\,\sigma(x,ds)\bigg)dx.
\end{split}
\end{equation}
Hence, if $n \neq m$, then $$(g_n,g_m)_H =0,$$ which implies that the linear spaces $\{\mathfrak L_n\}_{n=0}^\infty$ are
mutually orthogonal in $H$. Denote by $H_n$ the closure of $\mathfrak L_n $ in $H$. Then by Lemma \ref{giyr8i8},
$H =\bigoplus_{n=0}^\infty H_n$.

By \eqref{ns10}, setting $n =m$, we  get
\begin{equation}\label{ns11}
\|g_n\|_{H_n}^2  = \int_{\mathbb {R}^d}f_n^2(x)\,\bigg(\int_\mathbb {R} q^{(n)} (x,s)^2\,\sigma(x,ds)\bigg)dx = \int_{\mathbb {R}^d}f_n^2(x) \rho_n(dx),
\end{equation}
where $$\rho_n(dx) = \bigg(\int_\mathbb {R} q^{(n)} (x,s)^2\,\sigma(x,ds)\bigg)dx$$ is a measure on $(\mathbb{R}^d, \mathcal{B}(\mathbb{R}^d))$.
Consider a linear operator $$ \mathcal{D}\ni f_n\mapsto
\big(J_nf_n\big)(x,s) := f_n(x)q^{(n)}(x,s)\in
 \mathfrak L_n .$$
The image of $J_n$ is clearly the whole $\mathfrak L_n$. Now, $\mathfrak L_n$ is dense in $H_n$, while $\mathcal D$ is evidently dense in
$L^2(\mathbb R^d, \rho_n(dx))$. By \eqref{ns11}, for each $f_n\in\mathcal D$,
$$\|J_nf_n\|_{H_n} = \|f_n\|_{L^2(\mathbb R^d,\,\rho_n(dx))}.$$
Therefore, we can extend the operator $J_n$ by continuity to a unitary operator
\begin{equation}\label{ns12}
J_n: L^2(\mathbb R^d,\rho_n(dx)) \rightarrow H_n.
\end{equation}
In particular, $$H_n = \left\{f_n(x)\,q^{(n)}(x,s)\mid f_n \in L^2(\mathbb R^d, \rho_n(dx))\right\}.$$
Therefore, for each $k\geq 2$
\begin{equation*}
\begin{split}
H_n^{\otimes k} & = \bigg\{f_n^{(k)}(x_1,\ldots,x_k)\, q^{(n)}(x_1,s_1)\cdots q^{(n)}(x_k,s_k)\mid \\
&\quad f_n^{(k)} \in L^2(\mathbb R^d, \rho_n(dx))^{\otimes k}=L^2\big((\mathbb R^d)^k,\rho_n(dx_1)\cdots \rho_n(dx_k)\big)\bigg\}.
\end{split}
\end{equation*}
Since the operator $J_n$ in \eqref{ns12} is  unitary, we get that the operator
$$J_n^{\otimes k}:L^2(\mathbb R^d, \rho_n(dx))^{\otimes k} \rightarrow H_n^{\otimes k}$$ is also  unitary. The restriction of
$J_n^{\otimes k}$ to $L^2(\mathbb R^d, \rho_n(dx))^{\odot k}$ is a unitary operator
\begin{equation}\label{ns13}
J_n^{\otimes k}:L^2(\mathbb R^d, \rho_n(dx))^{\odot k} \rightarrow H_n^{\odot k}.
\end{equation}
Indeed, take any $f_n \in L^2(\mathbb R^d, \rho_n(dx))$. Then $f_n^{\otimes k} \in L^2(\mathbb R^d, \rho_n(dx))^{\odot k}$ and the set of all
such vectors is total in $L^2(\mathbb R^d, \rho_n(dx))^{\odot k}$. Now, by the definition of $J_n^{\otimes k}$, we get
\begin{align*}
J_n^{\otimes k}f_n^{\otimes k} = (J_n\,f_n)^{\otimes k}  \in H_n^{\odot k},
\end{align*}
and furthermore the set of all vectors of the form $(J_n\,f_n)^{\otimes k}$ is total in $H_n^{\odot k}$.
Hence, the statement follows.

For any $f_n^{(k)} \in L^2(\mathbb R^d, \rho_n(dx))^{\otimes k}$,
\begin{align*}
\big(J_n^{\otimes k}f_n^{(k)}\big)(x_1, s_1, \ldots, x_k, s_k) = f_n^{(k)}(x_1, \ldots , x_k) q^{(n)}(x_1,s_1) \cdots q^{(n)}(x_k,s_k).
\end{align*}
Hence, the unitary operator \eqref{ns13} acts as follows
\begin{align*}
\begin{split}
&L^2(\mathbb R^d, \rho_n(dx))^{\odot k} \ni f_n^{(k)}(x_1, \ldots , x_k)\\
& \quad \mapsto \big(J_n^{\otimes k}f_n^{(k)}\big)(x_1, s_1, \ldots, x_k, s_k) = f_n^{(k)}(x_1, \ldots , x_k) q^{(n)}(x_1,s_1) \cdots q^{(n)}(x_k,s_k).
\end{split}
\end{align*}
Thus, each function $g_n^{(k)} \in H_n^{\odot k}$ has a representation
$$g_n^{(k)} ( x_1,s_1, \ldots , x_k,s_k) = f_n^{(k)}(x_1, \ldots , x_k) q^{(n)}(x_1,s_1) \cdots q^{(n)}(x_k,s_k),$$ where
$f_n^{(k)} \in L^2(\mathbb R^d, \rho_n(dx))^{\odot k}$ and
$\|g_n^{(k)}\|_{H_n^{\odot k}} = \|f_n^{(k)}\|_{L^2(\mathbb R^d, \rho_n(dx))^{\odot k}}$.

For each $\alpha \in \mathbb Z_{+,0}^\infty$, we consider the Hilbert space
\begin{align}\label{by4.1}
L_\alpha^2\big((\mathbb R^d)^{|\alpha|}\big):=
L^2(\mathbb R^d, \rho_0(dx))^{\odot \alpha_0} \otimes L^2(\mathbb R^d, \rho_1(dx))^{\odot \alpha_1} \otimes \cdots.
\end{align}
We now define a unitary operator
$$J_\alpha: L_\alpha^2\big((\mathbb R^d)^{|\alpha|}\big) \rightarrow H_0^{\odot \alpha_0} \otimes H_1^{\odot \alpha_1} \otimes \dotsm,$$
where $$J_\alpha = J_0^{\otimes \alpha_0} \otimes J_1^{\otimes \alpha_1} \otimes \cdots .$$
We evidently have, for each $f_\alpha \in L_\alpha^2\big((\mathbb R^d)^{|\alpha|}\big)$,
\begin{align*}
\begin{split}
&\big(J_\alpha \, f_\alpha\big)(x_1,s_1,x_2,s_2, \ldots , x_{|\alpha|},s_{|\alpha|})\\
& \quad = f_\alpha(x_1, x_2, \ldots , x_{|\alpha|}) q^{(0)}(x_1,s_1)\cdots q^{(0)}(x_{\alpha_0},s_{\alpha_0})\\
&\quad \times   q^{(1)}(x_{\alpha_0 +1},
s_{\alpha_0+1})\cdots q^{(1)}(x_{\alpha_0 + \alpha_1},s_{\alpha_0+\alpha_1}) \cdots.
\end{split}
\end{align*}
For each $\alpha \in \mathbb{Z}_{+,0}^\infty$, we define a Hilbert space
$$\mathcal G_\alpha: = L_\alpha^2\big((\mathbb R^d)^{|\alpha|}\big) \alpha_0!\alpha_1! \cdots .$$
The  $J_\alpha$ is evidently a unitary operator
$$J_\alpha: \mathcal G_\alpha \rightarrow (H_0^{\odot \alpha_0} \otimes H_1^{\odot \alpha_1} \otimes \cdots) \alpha_0! \alpha_1! \cdots.$$
Denote
$\mathcal G:=  \bigoplus_{\alpha \in \mathbb Z_{+,0}^\infty} \mathcal G_\alpha$. Hence, we can construct a unitary operator
$$ J: \mathcal G \to \bigoplus_{\alpha\in\mathbb Z_{+,0}^\infty} (H_0^{\odot \alpha_0} \otimes H_1^{\odot \alpha_1} \otimes \cdots) \alpha_0! \alpha_1! \cdots$$
by setting $J:= \bigoplus_{\alpha\in\mathbb Z_{+,0}^\infty}  J_\alpha$.
By Lemma \ref{gfut7r7}, we get a unitary operator
$\mathcal R: \mathcal G \to\mathcal F(H)$,
by setting
$\mathcal R:=\Sym J$.  Thus, by Theorem \ref{ftydyr7rdaufc}, we get

\begin{theorem}\label{abcd26}
Let condition (A) be satisfied.
 We have a unitary isomorphism
$
\mathcal{K}: \mathcal{G} \rightarrow L^2(\mathcal{D}', \mu)
$
given by
$ \mathcal{K}:=I\mathcal R$,
where the unitary operator $I: \mathcal{F}(H) \rightarrow L^2(\mathcal{D}', \mu)$ is from Theorem \ref{ftydyr7rdaufc}.

\end{theorem}

\section{The unitary isomorphism $\mathcal K$ through multiple\\ stochastic integrals}

We will now give an interpretation of the unitary isomorphism $\mathcal K$ in terms of multiple stochastic integrals.
We will only present a sketch of the proof, omitting some technical
details.

Let us recall  the operators $A(\varphi)$ in $\mathcal F(H)$ defined by \eqref{rdtsw6u4ebg}.
Now, for each $k \in \mathbb N$, we define operators
\begin{equation}\label{tydy7erxx}A^{(k)}(\varphi): = a^+ (\varphi\otimes m_{k-1}) + a^0(\varphi\otimes m_{k}) + a^- (\varphi\otimes m_{k-1}).\end{equation}
In particular, $A^{(1)}(\varphi) = A(\varphi)$. 
The operator $A^{(k)}(\varphi)$ being symmetric, we denote by 
$A^{(k)}(\varphi)^\sim$ the closure of $A^{(k)}(\varphi)$.
For each $k\in\mathbb N$ and $\varphi \in \mathcal D$, we define
$Y^{(k-1)} (\varphi): = I (\varphi\otimes m_{k-1})$.
It can be shown that, for each $k\in\mathbb N$, {\it 
$IA^{(k)}(\varphi)^\sim I^{-1}$
is the operator of multiplication by the function $Y^{(k-1)}$.}

Suppose, for a moment, that the measures $\sigma(x,ds)$ do not depend on $x \in \mathbb R^d$.
For a fixed $\varphi \in \mathcal D$, let us orthogonalize in $L^2(\mathcal D', \mu)$ the functions $(Y^{(k)} (\varphi))_{k = 0}^\infty$.
This is of course equivalent to the orthogonalization of the monomials $(s^k)_{k=0}^\infty$ in $L^2(\mathbb R, \sigma)$. Denote by
$(q^{(k)})_{k=0}^\infty$ the system of monic orthogonal polynomials with respect to the measure $\sigma$. Denote
$(\varphi\otimes q^{(k)})(x,s):=\varphi(x)q^{(k)}(s)$.
Thus, the random variables
$$Z^{(k)}(\varphi): = I(\varphi\otimes  q^{(k)}),\quad k\in\mathbb Z_+,$$
appear as a result of the
orthogonalization of $(Y^{(k)}(\varphi))_{k=0}^\infty$. Since $q^{(0)}(s) = 1$, we have
$$Z^{(0)}(\varphi) = Y^{(0)}(\varphi) = \langle \cdot,\varphi  \rangle.$$
For each $k \geq 1$, we have a representation of $q^{(k)}(s)$ as follows:
$$q^{(k)}(s) = \sum_{i=0}^k b_i^{(k)}\,s^i.$$
Thus,
\begin{equation*}
Z^{(k)}(\varphi) = I(\varphi\otimes q^{(k)})  
= \sum_{i=0}^k b_i^{(k)} I(\varphi\otimes m_i)  = \sum_{i=0}^k b_i^{(k)} Y^{(i)} (\varphi).
\end{equation*}
Hence, under $I^{-1}$, the image  of the operator of multiplication by $Z^{(k)}(\varphi)$ is the operator
\begin{align*}
R^{(k)}(\varphi): &= \sum_{i=0}^k b_i^{(k)} (a^+(\varphi\otimes m_i) + a^- (\varphi\otimes m_i) + a^0 (\varphi\otimes m_{i+1}))\\
&= a^+(\varphi\otimes q^{(k)}) + a^- (\varphi\otimes q^{(k)}) + a^0 (\varphi\otimes\rho^{(k)}),
\end{align*}
where $\rho^{(k)}(s):=sq^{(k)}(s)$.

Let us now consider the general case, i.e., the case where the measure $\sigma(x, ds)$ does depend on $x \in \mathbb R^d$.
We are using the monic polynomials $(q^{(k)}(x,\cdot))_{k=0}^\infty$ which are orthogonal with respect to the measure $\sigma(x,ds)$. We have
$$q^{(k)}(x,s) = \sum_{i=0}^k b_i^{(k)}(x)\,s^i.$$
We  define 
\begin{equation*}
Z^{(k)}(\varphi): = I(\varphi q^{(k)})  
=  \sum_{i=0}^k Y^{(i)} (\varphi b_i^{(k)}) ,
\end{equation*}
where $(\varphi q^{(k)})(x,s):=\varphi(x)q^{(k)}(x,s)$. 
Hence, under $I^{-1}$, the image  of the operator of multiplication by $Z^{(k)}(\varphi)$ is the operator
\begin{align*}
R^{(k)}(\varphi): &= \sum_{i=0}^k  \big(a^+((\varphi b_i^{(k)})\otimes m_i ) + a^- ((\varphi b_i^{(k)})\otimes m_i ) + a^0 ((\varphi b_i^{(k)})\otimes m_{i+1} )\big)\\
&= a^+\bigg(\bigg(\varphi\sum_{i=0}^k b_i^{(k)}\bigg)\otimes m_i \bigg) + a^-\bigg(\bigg(\varphi\sum_{i=0}^k b_i^{(k)}\bigg)\otimes m_i \bigg) \\
& \quad + a^0\bigg(\bigg(\varphi\sum_{i=0}^k b_i^{(k)}\bigg)\otimes m_{i+1} \bigg)\\
&= a^+(\varphi q^{(k)}) + a^-(\varphi q^{(k)}) + a^0(\varphi \rho^{(k)}),
\end{align*}
where $\rho^{(k)}(x,s):=sq^{(k)}(x,s)$.

It is not hard to see that the above definitions and formulas can be easily extended to the case where the function $\varphi:\R^d\to\R$ is just measurable, bounded, and has compact support. In particular, for each $\Delta\in\mathcal B_0(\R^d)$, we will use the operators $Z^{(k)}(\Delta):=Z^{(k)}(\chi_\Delta)$.

We will now introduce a multiple Wiener--It\^o integral with respect to $Z^{(k)}$'s. So, we fix any $\alpha \in \mathbb{Z}_{+,0}^\infty$, $|\alpha|=n$, $ n \in \mathbb{N}$. Take any $\Delta_1,\dots, \Delta_n \in \mathcal{B}_0(\mathbb{R}^d)$, mutually disjoint. Then we define
\begin{align*}
&\int_{\Delta_1 \times \Delta_2 \times \cdots \times \Delta_n} dZ^{(0)}(x_1) \cdots dZ^{(0)}(x_{\alpha_0})dZ^{(1)}(x_{\alpha_0 + 1}) \cdots dZ^{(1)}(x_{\alpha_0 + \alpha_1})\\
&\quad\times dZ^{(2)}(x_{\alpha_0 + \alpha_1 +1}) \cdots\\
&=  \int_{(\mathbb{R}^d)^n} \chi_{\Delta_1}(x_1)\chi_{\Delta_2}(x_2) \cdots \chi_{\Delta_n}(x_n)dZ^{(0)}(x_1) \cdots dZ^{(0)}(x_{\alpha_0})\\
&\quad\times dZ^{(1)}(x_{\alpha_0 + 1}) \cdots dZ^{(1)}(x_{\alpha_0 + \alpha_1})dZ^{(2)}(x_{\alpha_0 + \alpha_1 +1}) \cdots\\
& :=Z^{(0)}(\Delta_1) \cdots Z^{(0)}(\Delta_{\alpha_0})Z^{(1)}(\Delta_{\alpha_0+1}) \cdots Z^{(1)}(\Delta_{\alpha_0+ \alpha_1})Z^{(2)}(\Delta_{\alpha_0+\alpha_1+1}) \cdots \, .
\end{align*}
Using the fact that the sets $\Delta_1, \dots ,\Delta_n$ are mutually disjoint,
\begin{align*}
&I^{-1}(Z^{(0)}(\Delta_1) \cdots Z^{(0)}(\Delta_{\alpha_0})Z^{(1)}(\Delta_{\alpha_0+1}) \cdots Z^{(1)}(\Delta_{\alpha_0+ \alpha_1})Z^{(2)}
(\Delta_{\alpha_0+\alpha_1+1}) \cdots )\\
&=R^{(0)}(\chi_{\Delta_1}) \cdots R^{(0)}(\chi_{\Delta_{\alpha_0}})R^{(1)}(\chi_{\Delta_{\alpha_0+1}}) \cdots R^{(1)}(\chi_
{\Delta_{\alpha_0+ \alpha_1}})R^{(2)} (\chi_{\Delta_{\alpha_0+\alpha_1+1}}) \cdots\\
&=a^+(\chi_{\Delta_1}q^{(0)}) \cdots a^+(\chi_{\Delta_{\alpha_0}} q^{(0)}) a^+(\chi_{\Delta_ {\alpha_0+1}} q^{(1)})\cdots
 a^+(\chi_{\Delta_{\alpha_0 + \alpha_1}} q^{(1)}) \\
&\quad\times a^+(\chi_{\Delta_{\alpha_0+\alpha_1+1}} q^{(2)}) \cdots \Omega \\
&=(\chi_{\Delta_1} q^{(0)}) \odot  \cdots \odot (\chi_{\Delta_{\alpha_0}} q^{(0)}) \odot (\chi_{\Delta_ {\alpha_0+1}}
q^{(1)}) \odot \cdots \odot (\chi_{\Delta_{\alpha_0 + \alpha_1}} q^{(1)}) \\
&\quad\odot (\chi_{\Delta_{\alpha_0+\alpha_1+1}} q^{(2)})\odot \cdots\\
&=\Sym_n\bigg(\bigg[(\chi_{\Delta_1} q^{(0)}) \odot  \cdots \odot (\chi_{\Delta_{\alpha_0}} q^{(0)})\bigg] \otimes  \bigg[(\chi_{\Delta_ {\alpha_0+1}}
 q^{(1)}) \odot \cdots \\
&\quad\odot (\chi_{\Delta_{\alpha_0 + \alpha_1}} q^{(1)})\bigg] \otimes  \cdots\bigg)\\
&= \Sym_n\bigg(\bigg[(\chi_{\Delta_1} \odot \cdots \odot \chi_{\Delta_{\alpha_0}}) (x_1, \ldots ,x_{\alpha_0})q^{(0)}(x_1,s_1) \cdots q^{(0)}(x_{\alpha_0},s_{\alpha_0})\bigg] \\
&\quad\otimes \bigg[(\chi_{\Delta_ {\alpha_0+1}} \odot \cdots \odot \chi_{\Delta_{\alpha_0 + \alpha_1}})
(x_{\alpha_0 + 1}, \ldots , x_{\alpha_0 + \alpha_1})\,q^{(1)}(x_{\alpha_0 +1},s_{\alpha_0+1})\\
&\quad\cdots q^{(1)}(x_{\alpha_0 + \alpha_1}, s_{\alpha_0 + \alpha_1})\bigg] \otimes \cdots\bigg)\\
&=\mathcal{R}\bigg((\chi_{\Delta_1} \odot \cdots \odot \chi_{\Delta_{\alpha_0}}) \otimes (\chi_{\Delta_ {\alpha_0+1}} \odot
\cdots \odot \chi_{\Delta_{\alpha_0 + \alpha_1}}) \otimes \cdots \bigg).
\end{align*}
Hence
\begin{align*}
&Z^{(0)}(\Delta_1) \cdots Z^{(0)}(\Delta_{\alpha_0})Z^{(1)}(\Delta_{\alpha_0+1}) \cdots Z^{(1)}(\Delta_{\alpha_0+ \alpha_1})Z^{(2)}
(\Delta_{\alpha_0+\alpha_1+1}) \cdots \\
&=\mathcal{K}((\chi_{\Delta_1} \odot \cdots \odot \chi_{\Delta_{\alpha_0}}) \otimes (\chi_{\Delta_ {\alpha_0+1}} \odot \cdots \odot
\chi_{\Delta_{\alpha_0 + \alpha_1}}) \otimes \cdots ).
\end{align*}
The set of all vectors of the form $$((\chi_{\Delta_1} \odot \cdots \odot \chi_{\Delta_{\alpha_0}}) \otimes (\chi_{\Delta_ {\alpha_0+1}} \odot
\cdots \odot \chi_{\Delta_{\alpha_0 + \alpha_1}}) \otimes \cdots )$$
is total in $\mathcal G_\alpha$. Therefore, by linearity and continuity,
we can extend the definition of the multiple Winner--It\^o integral to the whole space $\mathcal G_\alpha$. Thus, we get, for each $f_\alpha \in \mathcal{G}_\alpha$,
\begin{align*}
&\int_{(\mathbb{R}^d)^{|\alpha|}} f_\alpha (x_1, \dots , x_{|\alpha|})dZ^{(0)}(x_1) \cdots dZ^{(0)}(x_{\alpha_0}) dZ^{(1)}(x_{\alpha_0 +1}) \cdots dZ^{(1)}(x_{\alpha_0 +\alpha_1})\\
& \times dZ^{(2)}(x_{\alpha_0+\alpha_1+1}) \cdots = \mathcal{K} f_\alpha.
\end{align*}

Thus, we have the following theorem.
\begin{theorem}
 The unitary isomorphism $\mathcal{K}:\mathcal G \rightarrow L^2(\mathcal{D}',\mu)$ from Theorem~\ref{abcd26} is given by
\begin{align*}
&\mathcal G=\bigoplus_{\alpha \in \mathbb{Z}_{+,0}^\infty} \mathcal {G}_\alpha \ni (f_\alpha)_{\alpha \in {\mathbb{Z}_{+,0}^\infty}} = f \mapsto \mathcal{K}f\\
& \quad  = \sum_{\alpha \in \mathbb{Z}_{+,0}^\infty}\int_{(\mathbb{R}^d)^{|\alpha|}} f_\alpha (x_1, \dots , x_{|\alpha|})dZ^{(0)}(x_1) \cdots dZ^{(0)}(x_{\alpha_0}) \\
&\qquad \times dZ^{(1)}(x_{\alpha_0 +1}) \cdots dZ^{(1)}(x_{\alpha_0 +\alpha_1})dZ^{(2)}(x_{\alpha_0+\alpha_1+1}) \cdots.
\end{align*}
\end{theorem}

\begin{center}
{\bf Acknowledgements}\end{center}
 E.L. acknowledges the financial support of the Polish
National Science Center, grant no.\ Dec-2012/05/B/ST1/00626, and of the SFB 701 ``Spectral
structures and topological methods in mathematics'', Bielefeld University.
The authors are grateful to the anonymous referee for many useful suggestions.

\end{document}